\documentclass[12pt,a4paper]{amsart}

\usepackage{amscd}
\usepackage{graphicx}
\usepackage{amssymb,amsfonts,latexsym}
\usepackage{geometry}

\input xy
\xyoption {all}

\newtheorem{thm}{Theorem}[section]
\newtheorem*{thmA}{Theorem A}
\newtheorem*{corB}{Corollary B}
\newtheorem{cor}[thm]{Corollary}
\newtheorem{lem}[thm]{Lemma}
\newtheorem{prop}[thm]{Proposition}
\theoremstyle{definition}

\newtheorem{defn}[thm]{Definition}
\newtheorem*{quesG1}{Group-theoretic Question 1}
\newtheorem*{quesG2}{Group-theoretic Question 2}
\newtheorem*{quesA}{Arithmetic Question}
\theoremstyle{remark}

\numberwithin{equation}{section}
\newcommand{\dQ}{\mathbb{Q}}

\newcommand{\cB}{\mathcal{B}}

\newcommand{\fp}{\frak{p}}

\long\def\forget#1\forgotten{}

\begin{document}

\author{Daniela Bubboloni}
\address{Department of Economics and Management, University of Firenze, via delle Pandette 9, 50127 Firenze, Italy}
\email{daniela.bubboloni@unifi.it}

\author{Jack Sonn}
\address{
Department of Mathematics\\
Technion --- Israel Institute of Technology\\
Haifa, 32000\\
Israel }
\email{sonn@math.technion.ac.il}

\title[ ]{Intersective $S_n$ polynomials with few irreducible factors}
\date{\today}

\maketitle
\begin{abstract}
An intersective  polynomial is a monic polynomial in one variable with rational integer coefficients, with
  no rational root and having a root modulo $m$ for all positive integers $m$.
 Let $G$ be a finite noncyclic group and let $r(G)$ be the smallest number of irreducible factors of an intersective
 polynomial with Galois group $G$ over $\dQ$.
Let $s(G)$ be smallest number of proper subgroups of $G$ having the property that the union of their conjugates is $G$ and the intersection of all their conjugates is trivial. It is known that $s(G)\leq r(G).$
It is also known that if $G$ is realizable as a Galois group over the rationals, then it is also realizable  as the Galois group of an intersective polynomial.  However it is not known, in general, whether there exists such a polynomial  which is a product of the smallest feasible number $s(G)$ of irreducible factors. In this paper, we study the case $G=S_n$, the symmetric group on $n$ letters. 
  We prove that for every $n$, either $r(S_n)=s(S_n)$ or $r(S_n)=s(S_n)+1$ and that the optimal value $s(S_n)$ is indeed attained  for all odd $n$ and for some even $n$. Moreover, we compute $r(S_n)$ when $n$ is the product of at most two odd primes and we give general upper and lower bounds for $r(S_n).$
\end{abstract}
\section{Introduction}\label{sec:intro} \rm \vskip 2em
  An \it intersective \rm polynomial is a monic polynomial in one variable with rational integer coefficients, with
  no rational root and having a root modulo $m$ for all positive integers $m$, or equivalently, having a root in
$\dQ_p$ for all (finite) $p$.
 Let $G$ be a finite noncyclic group and let $r(G)$ be the smallest number of irreducible factors of an intersective
 polynomial with Galois group $G$ over $\dQ$. There is a group-theoretically defined lower bound for $r(G)$, given by the
 smallest number $s(G)$ of proper subgroups of $G$ having the property that the union of the
 conjugates of those subgroups is $G$ and their intersection is trivial.  This follows from
 \vskip 0.5em
 \begin{prop} \label{thm:char} {\em (\cite [Prop. 2.1]{so})}
Let $K/\dQ$ be a finite
Galois extension with Galois group $G$.  The following are
equivalent:

(1) \ \ $K$ is the splitting field of a product $f=g_1
\cdots g_m$ of $m$ irreducible polynomials of degree greater than
$1$ in $\dQ[x]$ and $f$ has a root in $\dQ_p$ for all (finite) primes $p$.

(2) \ \  $G$ is the union of the conjugates of $m$ proper
subgroups $A_1,...,A_m$, the intersection of all these conjugates
is trivial, and for all (finite) primes $\fp$ of $K$, the decomposition
group $G(\fp)$ is contained in a conjugate of some $A_i$.
\end{prop}

Recall that the decomposition group $G(\fp)$ is the stabilizer in $G$ of the prime $\fp$. Note also that we have  $2\leq s(G)\leq r(G)$, the first inequality holding because no group is the union of the conjugates of a single proper subgroup. The second inequality holds by definition.

It is natural to ask for which $G$, realizable over $\dQ$, is $s(G)=r(G)$?  This paper focuses on this question for the symmetric groups $G=S_n$ of degree $n\geq 3$, which are well known to be realizable over $\dQ$. We view $S_n$ as naturally acting on the set $\Omega=\{1,\dots,n\}.$

If we drop the trivial intersection
 condition in the definition of $s(G)$, we obtain the {\it normal covering number} of $G$, denoted
 by $\gamma(G)$ in the group theory literature. We recall some terminology from  \cite{bp}.
  If $H_1,\ldots,H_l$, with $l\in\mathbb{N}$,
are pairwise non-conjugate proper subgroups of $G$ such that
$G=\bigcup_{g\in G} \bigcup_ {i=1}^lH_i^g,$
we say that $\Delta=\{H_i^g\,|\,1\leq i\leq l, g\in G\}$ is a {\em
 normal covering} of $G$ and that $\delta=\{H_1,\dots,H_l\}$ is a
{\em basic set} for $G$ generating $\Delta.$ We call the elements
of $\Delta$ the {\em components}  and
the elements of $\delta$ the {\em basic components} of the normal covering $\Delta$.
The minimum cardinality $\gamma(G)$ of a basic set
is called the normal covering number of $G$.
Recall that $\gamma(G)\geq 2$ and note that $\gamma(G)$ can be seen as the minimum number of proper subgroups of $G$ such that every cyclic subgroup of $G$ lies in some conjugate of one of them.
By definition, we have $\gamma(G)\leq s(G)$. On the other hand,  if the trivial intersection condition does not hold for a set of
 subgroups whose conjugates cover $G$, adding the trivial subgroup restores the trivial intersection property and thus, for every finite group $G$, we have $s(G)\in \{\gamma(G),\ \gamma(G)+1\}$. In particular, for the symmetric groups, it is easily seen that
 $s(S_n)=\gamma(S_n)$ (Lemma \ref{s-gamma}).  We then accordingly ask whether, for every $n\geq 3,$
 $\gamma(S_n)=r(S_n)$.

 Fortunately, much is already known about $\gamma(S_n)$.  In \cite[Theorem 1.1] {JA}, it is
 proved that $\gamma$ grows linearly with $n$, in the sense that there exists $k\in \mathbb{R}$, with   $0<k\leq 2/3$, such that
 \begin{equation}\label{linear} kn \leq \gamma(S_n) \leq 2n/3
 \end{equation}
 for all $n\geq 3$.
 In \cite {bp} exact values of $\gamma(S_n)$ are
 given for all $n$ odd and divisible by at most two distinct
 primes.  It has been shown for the case $\gamma(S_n)=2$, which holds
 exactly for $3\leq n \leq 6$, that there exist Galois realizations for
 which $r(S_n)=2$ \cite {rs}.  The present paper  gives for the first time an infinite set of $n$ for
 which $r(S_n)=\gamma(S_n)$. In fact, we show that this holds for all odd $n$ (Proposition \ref{odd}).

 To state our first main result, we  need to introduce a class of metacyclic subgroups  of $S_n$, which are in fact abelian on two generators.  For $m\in \mathbb{N},$ denote by $C_m$ the cyclic group of order $m.$ Let $M$ be a subgroup  of $S_n$ of the form $C_{2m}\times C_2$, for some $m\in\mathbb{N}$, with $C_2$ generated by a transposition $\tau=(i\ j)$. We call $M$ a {\it special metacyclic} subgroup  of $S_n$ and denote by $\mathcal{M}(S_n)$ the set of special metacyclic subgroups  of $S_n$. 
 We claim that the factor $C_{2m}$ in $M\in \mathcal{M}(S_n)$ can be chosen generated by $\sigma\in S_n$ such that $\sigma(i)=i$ and $\sigma(j)=j$. 
Indeed, let $C_{2m}=\langle \psi\rangle$, $C_2= \langle \tau\rangle$ and $M=\langle \psi\rangle\times \langle \tau\rangle$. Note that $|M|=4m.$
Then $C_{2m}=\langle \psi\rangle$ is contained in the centralizer of $\tau$ in $S_n$, which is the direct product of  $\langle\tau\rangle$ and the subgroup $U$ of $S_n$ fixing $i$ and $j$. In particular, we have $\psi=\sigma$ or $\psi=\sigma\tau$, for some $\sigma\in U.$ In the first case our claim is obvious. In the second case, we have $M=\langle \sigma\tau\rangle \times \langle \tau\rangle=\langle \sigma\rangle \times \langle \tau\rangle$ and $4m=2|\sigma|$ gives $|\sigma|=2m,$
 so that $\sigma$ generates a cyclic group of order $2m.$
Throughout the paper the two generator $\sigma$ and $\tau$ of $M$ will be always chosen such that  $\tau=(i\ j)$, $\sigma(i)=i$ and $\sigma(j)=j$.

 \begin{thmA} \label{thm:main} For any $n$, the symmetric group $S_n$  is realizable infinitely
often as a Galois group over $\dQ$, with all decomposition groups either cyclic or special metacyclic of the form $C_{2m}\times C_2$,
where the factor $C_2$ generated by a transposition is the inertia group.  The factor $C_{2m}$ can be chosen so that it fixes the two letters moved by the transposition.
 \end{thmA}

Theorem A suggests the definition of a further useful parameter  $\gamma'(S_n).$   We call a basic set $\delta'$ {\it special} if
every subgroup of $S_n$ which is either cyclic or special metacyclic, is contained in a conjugate of a component in $\delta'$. If $\delta'=\{H_1,\dots, H_l\}$ is a special basic set,  we call the corresponding covering $\Delta'=\{H_i^g\,|\,1\leq i\leq l, g\in S_n\}$  a {\em special
  normal covering} of $S_n.$
The minimum cardinality of a  special basic set for $G$
is called the \emph{special normal covering number} of $S_n$ and denoted by $\gamma'(S_n);$
a special basic set of size $\gamma'(S_n)$ is called a {\it minimal special basic set}.
Note that if $n=3$, no special metacyclic subgroup exists and thus,  $\gamma'(S_3)=\gamma(S_3)=2.$

For $x\in \mathbb{N}$, with $1\leq x\leq n/2$, consider the intransitive subgroup of $S_n$, defined by $P_x=\{\psi\in S_n : \psi(\{1,\dots,x\})=\{1,\dots,x\}\}.$  Clearly, if $X\subseteq \Omega$ has size $c\in\mathbb{N}$, for some $ c<n,$
and $G=\{\psi\in S_n : \psi(X)=X\}$, then there exists $g\in S_n$ such that $G=P_x^g,$ where $x=\min\{c,n-c\}$.  Moreover, the set of maximal subgroups of $S_n$ which are
intransitive is given by the conjugates of the subgroups in
$\mathcal{P}$, where
\begin{equation}\label{P}
\mathcal{P}=\{\ P_x: 1\leq x<n/2\}.
\end{equation}
Recall that, for $n$ even, the intransitive subgroup $P_{n/2}$
is not maximal in $S_n$ because it is properly contained in the maximal imprimitive subgroup $S_{n/2}\wr S_2$ (see Section \ref{basic}).

 Consider $M\in \mathcal{M}(S_n)$,  with generators  $\sigma$ and $\tau=(i\ j)$, so that $\sigma(i)=i$ and $\sigma(j)=j$. Then, for $X=\{i,j\}$, we have  $\psi(X)=X$ for all $\psi\in M$ and so, up to conjugacy, $M$ is contained in $P_2$. It follows that $\gamma(S_n)\leq \gamma'(S_n)\leq \gamma(S_n)+1,$ and $\gamma'(S_n)=\gamma(S_n)$ if there exists a minimal normal
covering of $S_n$ admitting as component $P_2$ or some proper overgroup $H$ of $P_2.$ Note that,  $P_2$ being maximal in $S_n$ for $n\geq 5$, this  last possibility can happen only for $n=4$, through $H=S_2\wr S_2.$
Interestingly, there are examples of $n$ for which $\gamma'(S_n)=\gamma(S_n)$ even though no minimal normal covering of $S_n$ has a component containing $P_2$ (Proposition \ref{10,14}).

\vskip .5em
Combining Proposition 1.1 with Theorem A, we easily obtain the following interesting corollary.

\begin{corB}\label{prop:prop}   Let $n\in \mathbb{N}, n\geq 3.$  Then $2\leq \gamma(S_n)\leq  r(S_n) \leq\gamma'(S_n)$.  In particular,
$r(S_n)$ equals $\gamma(S_n)$ or $\gamma(S_n)+1$, and equals $\gamma(S_n)$ if there exists a minimal normal
covering of $S_n$ which includes $P_2$.
\end{corB}

Theorem A and Corollary B naturally raise two questions.
 \vskip .5em
\begin{quesA}\label{Arithmetic Question}{\em  Is $S_n$ realizable over the rationals $\dQ$
 with all decomposition groups cyclic?}
 \end{quesA}

 \rm The answer to this question appears to be unknown.  When the
 answer is yes for a given $n$, it is immediate also that
 $r(S_n)=\gamma(S_n)$.

\begin{quesG1}\label {Group1}{\em Is $\gamma'(S_n)=\gamma(S_n)$ for
 all $n\in \mathbb{N}, n\geq 3$?}
 \end{quesG1}
An affirmative answer to this question also gives $r(S_n)=\gamma(S_n)$. In Section \ref{5}, we give an affirmative answer for $n$ odd (Proposition \ref{odd}) but the question remains open, in general, for $n$ even. An indication of the complexity for the even case is illustrated by the cases $n=10$ and $n=14,$ which we treat in Section \ref{10}.
Some support for an affirmative answer in the general case might be given by the fact that any known upper bound for $\gamma(S_n)$ holds also for $\gamma'(S_n)$ (Propositions \ref{gamma'-g} and \ref{h-bound}). Moreover, for all the $n$ such that the value of $\gamma(S_n)$ is known, we have $\gamma'(S_n)=\gamma(S_n).$

 We conclude the paper by finding for $r(S_n)$ as well as for $\gamma'(S_n)$ the same linear bounds \eqref{linear} known for $\gamma(S_n)$ (Proposition \ref{bounds}).
 \vskip 1cm

\section{Intersective $S_n$ polynomials with few irreducible factors}\label{sec:one}

We start by noting that for the symmetric group, the parameters $s$ and $\gamma$ coincide:
\begin{lem}\label{s-gamma} If $\delta=\{H_1,\dots, H_k\}$ is a basic  set for $S_n$, then $\underset { \sigma\in S_n}{\cap} \,\overset{k}{\underset{i=1}{\cap}}H_i^{\sigma}=1.$ In particular $\gamma(S_n)=s(S_n).$
\end{lem}
\begin{proof}  Assume that $K=\underset { \sigma\in S_n}{\cap} \,\overset{k}{\underset{i=1}{\cap}}H_i^{\sigma}\neq 1.$ Since  $K\lhd S_n$, the only possibility is $K=A_n.$ Then,  for every $ i\in \{1,\dots,k\}$, we have $A_n\leq H_i<S_n, $ which gives $H_i=A_n$ and thus $\delta=\{A_n\}$, a contradiction.
Next let $\delta$ be a minimal basic set with $\gamma(S_n)$ components. By what is shown above, $\delta$ realises the trivial intersection property, thus $s(S_n)\leq \gamma(S_n)$. As  $s(G)\geq \gamma(G)$ for all  finite groups $G$, the equality $\gamma(S_n)=s(S_n)$  holds.
\end{proof}
\vskip 0.5em

The proof of Theorem A is based on a  construction of Kedlaya  (\cite{ke}) of
 infinitely many Galois realizations of the symmetric groups $S_n$ over $\dQ$ with squarefree discriminants, together
 with an earlier result of Kondo (\cite{ko}).
\begin{thm}\label{kedlaya}{\rm (Kedlaya)}
Let $n>1$ be an integer and let $S$ be a finite set of primes.  Then there exist infinitely many monic irreducible
polynomials $P(x)$ of degree $n$, with integer coefficients, such that the discriminant of $P(x)$ is squarefree and not
divisible by any of the primes in $S$.
\end{thm}
\noindent
\textit{Proof of Theorem A.}\quad Let $n\in \mathbb{N}, n\geq 3$ and $S=\{p\ \hbox{prime}: p\leq n\}$. Let $P(x)$ be a polynomial given by Theorem \ref{kedlaya} and let $K$ be its splitting field. Since, as pointed out by Kedlaya in \cite{ke}, citing Kondo \cite{ko}, an irreducible polynomial of degree $n$ with rational integer coefficients whose discriminant is squarefree has Galois group $S_n$ over $\dQ$, we have that $G(K/\dQ)=S_n.$
Let $p$ be a rational prime.  If $p$ is unramified in $K,$ then its decomposition group is cyclic.
We may therefore assume $p$ is ramified in $K$.  By Kondo \cite [Lemma 2, Theorem 2] {ko}, the inertia group of a prime $\fp$ of $K$ dividing $p$ is of order two and generated by a transposition $\tau$.
As $p$ divides the discriminant of $P(x)$, we have that $p>n.$ In particular, $p$ does not divide $n!=|S_n|$. Since the ramification index $e_p$ divides the order of the Galois group of $K/\dQ$, $p$ does not divide $e_p$,  so that $p$ is tamely ramified in $K$. Recall now that, for any prime, the inertia group is a normal subgroup of the decomposition group, and the quotient group is cyclic of order $f$, where $f$ is the inertia degree of the prime $\fp$ over $p$.  For a tamely ramified prime, the inertia group is cyclic as well, and if it is also of order $2$ as in our case, it is central. Thus the decomposition group is  metacyclic abelian.
Moreover, as $\tau$ is not a square in $S_n$, the decomposition group splits into a direct product of $\langle\tau\rangle$ and a
cyclic group $C=\langle\sigma\rangle$ of order $f$.  If $f$ is odd, then the decomposition group is cyclic.  If $f$ is even, then the decomposition group belongs to $\mathcal{M}(S_n)$  and the possibility to choose the factor $C$  fixing the two letters moved by $\tau$ is guaranteed by the discussion about the groups in $\mathcal{M}(S_n)$  made in the introduction.
\qed

\vskip .5em
\it Proof of Corollary B. \rm The inequality $\gamma(S_n)\leq r(S_n)$ follows from Lemma \ref{s-gamma}  recalling that, by Proposition \ref{thm:char}, $s(S_n)\leq r(S_n).$
To show  $r(S_n)\leq \gamma'(S_n)$, let $\delta'$ be a special basic set  of  $S_n$ of minimal cardinality $\gamma'(S_n)$.  By Theorem A, there exists an $S_n$-extension $K$ of $\dQ$ with decomposition groups normally covered by $\delta'$.  By Proposition \ref{thm:char}, there exists an intersective polynomial with splitting field $K$ which is a product of $\gamma'(S_n)$ irreducible factors.  Hence $r(S_n)\leq \gamma'(S_n)$.  The final assertion is immediate from the first, together with the relation between $\gamma(S_n)$ and $\gamma'(S_n)$  discussed in the introduction. \qed

\vskip 1cm

\section{Partitions and permutations}\label{partitions}

The next sections of the paper deal with the question of whether or not $\gamma'(S_n)=\gamma(S_n)$. To start with we need some definitions.

\subsection{Partitions and cuts} Let $n,k\in\mathbb{N}$, with $k\leq n$. A $k$-{\em partition} of $n$ is an unordered $k$-tuple
$T=[x_1,\dots,x_k]$, with $x_i\in\mathbb{N}$ for all $i\in
\{1,\dots,k\},$ such that $n=\sum_{i=1}^{k}x_i.$
The $x_i$ are called the {\em terms} of the $k$-partition and, obviously, we have $1\leq x_i\leq n.$ If $T$ is a $k$-partition of $n$, for some $k\in \mathbb{N}$, we say that $T$ is a partition of $n.$ We denote by $\mathcal{T}(n)$ the set of partitions of $n.$
Fix  $v\geq n$ and call $v$ the representation length for $\mathcal{T}(n)$. If $T\in \mathcal{T}(n)$, let $m_j\in \{0,1,\dots, n\}$ be the number of times in which $j\in \{1,\dots, v\}$ appears as a term in $T.$ We call $m_j$ the {\it multiplicity} of $j$ in $T$ and say that  $T=[1^{m_1},\ 2^{m_2}, \dots, v^{m_v}]$  is the representation of $T$ of length $v.$
Note that, within this representation, we have $\sum_{j=1}^v j\,m_j=n$ and that the exponents $m_j$ do not represent a power. Obviously, $m_j=0$ for all $n<j\leq v$. The multiplicities equal to $1$ are usually omitted.
In many contexts also  the multiplicities equal to $0$ are omitted,  but in others some of them can be usefully put in evidence.  Let $T=[1^{m_1},\ 2^{m_2}, \dots, n^{m_n}]\in \mathcal{T}(n) $ be represented with length $n$. Let, for every $j\in \{1,\dots,n\}$, $0\leq s_j\leq m_j$ be such that $c=\sum_{j=1}^n j\,s_j$ satisfies $0< c< n$. Then
$T'=[1^{s_1},\ 2^{s_2}, \dots, n^{s_n}]\in \mathcal{T}(c)$ is called a {\it subpartition} of $T.$ Note that $T'$ is also represented with length $n$.
Let  now $c\in \mathbb{N}$, with $0< c< n$ and use $n$ as a common representation length for $\mathcal{T}(n)$, $\mathcal{T}(c)$ and $\mathcal{T}(n-c)$. If
$T_1\in\mathcal{T}(c)$ is a subpartition of $T,$ then $T_1$  defines, in a natural way, the {\it complementary partition}  $T_2=[1^{m_1-s_1},\ 2^{m_2-s_2}, \dots, n^{m_n-s_n}]\in \mathcal{T}(n-c)$. We say that  $(T_1,T_2)$ realizes  a $c$-{\it cut } for $T$ and write $T=[T_1\mid T_2].$
 If $T=[T_1\mid T_2]$ is a cut for $T$, we say that the cut {\it isolates} $T'$ if $T'$ is a subpartition of $T_1$ or $T_2$. Note that if the $c$-cut  $[T_1\mid T_2]$ isolates $T'$, then also the $(n-c)$-cut $[T_2\mid T_1]$ isolates $T'$.

For instance, if $T=[1^3,2^2,5],\ n=12,\, c=7$ and  $T_1=[1^2,5]$, then $T_1$ is a  partition of $7$ which is a  subpartition of $T$, and
$T_2=[1, 2^2]$ is the complementary partition of $n-c=5.$ Thus $T=[1^2,5\mid 1,2^2]$ is a $7$-cut for $T,$  which isolates $T'=[1,5]$ as well as $T'=[2^2]$ but not $T'=[2,5].$

\subsection{The type of a permutation}\label{type}
Let $\sigma\in S_n$  and let $\mathcal{O}(\sigma)$ be the set of orbits of $\sigma$ in the natural action on $\Omega.$ Let $X_1,\dots, X_k$ be the distinct elements of $\mathcal{O}(\sigma)$ and put $x_i=|X_i|$. Then the unordered list  $T_{\sigma}=[x_1,\dots,x_k]$ is a $k$-partition of $n$, called the {\it  type} of $\sigma$.
Note that the fixed points of $\sigma$ correspond to the $x_i=1,$ while the
lengths of the disjoint cycles in which $\sigma$ splits are given by the $x_i\geq 2.$ Recall that  the order $|\sigma|$ of $\sigma$ may be recovered by $T_{\sigma}$ through $|\sigma|=\mathrm{lcm}\{x_i\}_{i=1}^k.$ In particular, if $|\sigma|$ is even, then at least one $x_i$ is even.
Clearly, for all $k\in \mathbb{N}$, with $1\leq k\leq n,$ each $k$-partition of $n$ is the type of some permutation in $S_n$. Therefore $\mathcal{T}(n)$ coincides with the set of types for $S_n$. The concept of type is crucial in dealing with normal coverings for the symmetric group, because
 $\sigma,  \nu\in S_n$ are conjugate in $S_n$ if and only if $T_{\sigma}=T_{\nu}$. Thus $\delta=\{H_1,\dots,H_l\}$ is a
basic set for $S_n$ if and only if for every $T\in \mathcal{T}(n)$ there exists $j\in \{1,\dots, l\}$ such that $H_j$ contains a permutation $\sigma$ with
$T_{\sigma}=T.$ When a subgroup $H$ of $S_n$ contains a permutation of
type $T$ we say that `$T$ belongs to $H$' and we write $T\in H.$

\subsubsection{The canonical form of $T_{\sigma}$}\label{sigma}
Let $M\in\mathcal{M}(S_n)$ with generators $\sigma$ and $\tau$. Recall that $\tau=(i\ j)$ is a transposition while $\sigma$ is a permutation with $|\sigma|$ even, $\sigma(i)=i$ and $\sigma(j)=j.$
 In particular, $T_{\tau}=[1^{n-2},2]$  and $T_{\sigma}=[1^2, x_1,\dots, x_k]$,  where $k=|\mathcal{O}(\sigma)|-2\geq 1$. Note that $1\leq x_i<n$, for all $1\leq i\leq k.$ We say that the type of $\sigma$ is represented in  canonical form  if $ x_1,\dots, x_k\in\mathbb{N}$ are arranged so that  there exists $s\in\mathbb{N}$,  with $s\leq k$ such that $ x_i$ is
 even for $1\leq i\leq s,$  while $  x_i$  is odd for $ s<i\leq k.$
 We set $m=k-s\geq 0.$  Then, $m=0$ means that the terms $x_i$ in $T_{\sigma}$ are even for all $1\leq i\leq k.$
\vskip 1cm

 \section{Maximal subgroups of $S_n$}\label{basic}
In order to determine $\gamma(S_n)$ or $\gamma'(S_n)$, we may obviously assume that the components of a
normal covering are maximal subgroups of $S_n.$ These subgroups may be intransitive,
primitive or imprimitive. The intransitive ones have been described in the introduction as the conjugates of the subgroups in $\mathcal{P}$, with $\mathcal{P}$ defined in \eqref{P}.

Let $n=bm,$ where $b\mid n$ and $2\leq b\leq n/2.$ If $\cB$ is a partition of $\Omega$ into $m$ subsets of size $b$, we say that $\cB$ is a $(b,m)$-block system
 for $\Omega$. The imprimitive maximal subgroups of $S_n$ are the stabilisers of  all the possible block systems. Consider, for $j\in\{0,\dots, m-1\}$, the $m$ proper subsets of $\Omega$  of size $b$ given by $B_j=\{jb+i: i\in \{1,\dots, b\}\}$. Then $\cB_0=\{B_j:j\in\{0,\dots, m-1\}\}$ is a particular $(b,m)$-block system for $\Omega$ and we denote by $S_b\wr S_{m}$  its stabiliser in $S_n$. Then
the set of imprimitive maximal subgroups of $S_n$ is obtained by the conjugates of the subgroups in the set  $
\mathcal{W}$, where
$$
\mathcal{W}=\{\ S_b\wr S_{m} \ : 2\leq b\leq n/2,\, b\mid n,\ m=n/b \}.
$$

The primitive maximal subgroups of $S_n$, different from $A_n$, do not play a significant role in the normal coverings and they are excluded in all the known minimal normal coverings, with the exception of  the case $n$ prime. Namely, for any prime $p\geq 5,$ the group $S_p$ admits a unique minimal normal covering generated by the basic set
$$
\delta=\{AGL_1(p)\cong C_p\rtimes C_{p-1}, \
P_k\  :\ 2\leq k\leq \frac{p-1}{2}\},
$$
admitting the primitive maximal component $AGL_1(p).$
In particular $\gamma(S_p)=\frac{p-1}{2}$ (\cite[ Proposition 7.1]{bp}).
Recall also that the unique minimal normal covering of $S_3$ is generated by the basic set $\{A_3, P_1\}.$ Since we have observed that $\gamma'(S_3)=\gamma(S_3)$, by Corollary B, we get
 \begin{equation}\label{3}
 \gamma'(S_3)=\gamma(S_3)=r(S_3)=2.
\end{equation}

\vskip 1cm

\section{Special normal coverings and normal coverings}\label{5}

In this section, we study Group-theoretic Question 1 exploring the link between $\gamma'(S_n)$ and $\gamma(S_n)$ for $n\in A$, where $A=\{n\in\mathbb{N}: n\geq 4\}$. We start by giving an affirmative answer in the  case $n$ odd and go on, in the general case, showing that all the known upper bounds for $\gamma(S_n)$ hold also for $\gamma'(S_n)$.

\subsection{The odd degree case}\label{P2}
\begin{prop}\label{odd} Let $n\in A$ be odd. Then every minimal basic set with maximal components contains $P_2$. In particular,
$\gamma'(S_n)=\gamma(S_n)=r(S_n).$
\end{prop}
\begin{proof}
Let $n\geq 5$ be odd and $\delta$ be a basic set for $S_n$, with maximal components. Consider the type $T=[2,n-2]$. Since $\gcd(2,n)=1$, by Lemma 5.2 in \cite{bp}, the only maximal subgroup of $S_n$ containing a permutation of  type $T$ is $P_2$. It follows that $P_2\in \delta$ and we conclude applying Corollary B.
\end{proof}

\subsection{The main bound $g$ }\label{main}

\begin{defn}\label{g}  For $n\in A$, let $\nu(n)$ be the number of the distinct prime  factors of $n$ and write

\begin{equation}\label{factors}
n=p_1^{\alpha_1}\cdots p_{\nu(n)}^{\alpha_{\nu(n)}}
\end{equation}
where, for every $i,j\in \{1,\dots, \nu(n)\}$,  $\alpha_i\in\mathbb{N}$, $p_i$ is a prime number and $p_i<p_j,$ for $i<j$.

Define the function $g:A\rightarrow \mathbb{N}$ by
\begin{equation}\label{g-definition}
g(n)=\begin{cases}
\ \frac{n}{2}(1-\frac{1}{p_1}) &\quad \mbox{if  } \nu(n)=1, \alpha_1=1\\

\ \frac{n}{2}(1-\frac{1}{p_1})+1         &\quad \mbox{if  } \nu(n)=1, \alpha_1\geq 2\\

\  \frac{n}{2}(1-\frac{1}{p_1}) (1-\frac{1}{p_2})  +1        &\quad \mbox{if}\   \nu(n)=2, (\alpha_1,\alpha_2)=(1,1)\\

\ \frac{n}{2}(1-\frac{1}{p_1}) (1-\frac{1}{p_2})  +2              &\quad \mbox{otherwise}\\
 \end{cases}
\end{equation}

\end{defn}

Note that, in the cases with $\nu(n)=1$,  $\frac{n}{2}(1-\frac{1}{p_1})$ counts the natural numbers  less than $n/2$ and not divisible by $p_1$; in the  cases with $\nu(n)\geq 2$, $\frac{n}{2}(1-\frac{1}{p_1}) (1-\frac{1}{p_2})$ counts  the natural numbers  less than $n/2$ and not divisible by either $p_1$ or $p_2$ ( \cite[Proposition 2.4]{bps}). Moreover, for every $n\in A$, $g(n)\geq 2.$ The function $g$ plays an important role in bounding  $\gamma(S_n)$.  When $\nu(n)$ is small and $n$ is odd, then $g(n)$ gives the exact value for $\gamma(S_n)$.
Namely, by \cite[ Proposition 3.1]{bps} and  \cite[ Propositions 7.1, 7.5, 7.6]{bp}  we have  the following.

\begin{prop}\label{gamma-g}  Let $n\in A$. Then:
\begin{itemize} \item[i)] $\gamma(S_n)\leq g(n)$, with equality when $n$ is odd and $\nu(n)\leq 2$;
\item[ii)] if $\nu(n)=2$ and $(\alpha_1,\alpha_2)\neq(1,1)$ or if $\nu(n)\geq 3$, then
\begin{equation}\label{basic-set}
\delta_C=\{P_x\ :\ 1\leq x<n/2,\ \gcd(x,p_1p_2)=1\}\cup\{S_{p_1}\wr S_{n/p_1},\ S_{p_2}\wr S_{n/p_2}\}
\end{equation}
is a basic set of order $g(n)$, called the canonical basic set.
\end{itemize}
\end{prop}

There are  other odd degree cases for which $\gamma(S_n)=g(n)$; for instance, by \cite[Theorem 1.1]{bps}, for $n=15q$, that holds when $q$ is an odd prime such that $q\equiv 2 \pmod{15}$ and $q\not\equiv 12 \pmod{13}$.
 Also, for all the even cases in which $\gamma(S_n)$ is known, we have $\gamma(S_n)=g(n)$; for instance, that holds for all even $n$, with $4\leq n\leq 12$ (see \cite[Table 1]{bps}). On the other hand, we stress that no general exact formula is known for $\gamma(S_n)$, when $n$ is even, with a lack of knowledge even when $n$ is a power of $2$ (see \cite[Problems 2 and 3] {bps}). The bound $\gamma(S_n)\leq g(n)$ is the best information we have on the number $\gamma(S_n),$ when $n$ is even (see also Sections \ref{Maroti} and \ref{even-case}).

We now want to obtain the same inequality of Proposition \ref{gamma-g} for $\gamma'(S_n).$ To that purpose, we need some preliminary results.

 \begin{lem}\label{wreathgen} Let $n\in\mathbb{N}$ and $b, m\in\mathbb{N}$ with $b,m\geq 2$ such that $n=bm.$ Let
 $l\in \mathbb{N}$ and, for $i\in \{1,\dots,l\} $, let $\sigma_i\in S_n$  be cycles with lengths divisible by $b.$ If the $\sigma_i$ are disjoint, then $\langle \sigma_1,\dots, \sigma_l \rangle$ is contained in a conjugate of $S_b \wr S_m$.
 \end{lem}

\begin{proof} By hypothesis, the length of $\sigma_i$ is of the form $m_ib$, for some $m_i\in\mathbb{N}$, and thus,  for each $i\in \{1,\dots,l\}$, $\sigma_i^{m_i}$ is a disjoint product of $b$-cycles. Let $\Gamma=\{j\in \Omega: \sigma_i(j)=j, \ \hbox{for all}\ i\in\{1,\dots,l\}\}$. From $b\mid n$, we get $b\mid |\Gamma|$, so that $\Gamma$ can be partitioned into $k=|\Gamma|/b\geq 0$ subsets $\Gamma_j$ of size $b,$ for $j\in\{1,\dots, k\}$.
The set of  orbits of the $\sigma_i^{m_i}$ over all $i\in \{1,\dots,l\}$  together with the sets $\Gamma_j$ over all $j\in\{1,\dots, k\}$, is thus a $(b,m)$-block system $\cB$ for $\Omega$. Now note that each $\sigma_i$ commutes with $\sigma_i^{m_i}$, as well as with each $\sigma_j$, $j\neq i$. Moreover each $\sigma_i$ fixes every $\Gamma_j.$
Thus each $\sigma_i$ stabilizes the block system $\cB$.  It follows that $\langle \sigma_1,\dots, \sigma_l \rangle$ is contained in the stabilizer of $\cB$ in $S_n$ and thus in a suitable conjugate of $S_b \wr S_m$.
\end{proof}
 \begin{cor}\label{wreathbis} Let $n\in\mathbb{N}$  and $\sigma\in S_n$. If $b\in\mathbb{N},$ with $b\geq 2,$ divides all the terms $x_i$ in
  $T_{\sigma}=[x_1,\dots,x_k]$, then  $\sigma$ is contained in a conjugate of $S_b \wr S_{n/b}$.
   \end{cor}
\begin{proof}  Split  $\sigma=\sigma_1\cdots \sigma_k$  into disjoint cycles $\sigma_i$ of length $x_i$, for $i\in \{1,\dots, k\}.$ By Lemma \ref{wreathgen}, $\langle \sigma_1,\dots, \sigma_l \rangle$ is contained in a conjugate of $S_b \wr S_{n/b}$. In particular $\sigma\in \langle \sigma_1,\dots, \sigma_l \rangle$ is contained in a conjugate of $S_b \wr S_{n/b}$.
\end{proof}

\begin{cor}\label{wreath} Let $n\in A$ and $M=\langle \sigma \rangle\times \langle \tau\rangle\in \mathcal{M}(S_n).$ If in the canonical form of $T_{\sigma}=[1^2, x_1,\dots, x_k]$ every $x_i$ is even, then $M\leq (S_2\wr S_{n/2})^g$, for some $g\in S_n$.
 \end{cor}
\begin{proof}  Split  $\sigma=\sigma_1\cdots \sigma_k$  into disjoint cycles $\sigma_i$ of length $x_i$, for $i\in \{1,\dots, k\}$ and define $\sigma_{k+1}=\tau$. Note that $n$ is necessarily even so that  $b=2\mid n$.
Now apply Lemma \ref{wreathgen} to the cycles $\sigma_i$ for $i\in \{1,\dots, k+1\}$ and to $b=2$.
\end{proof}

\begin{lem}\label{intrans} Let $n\in A$ and $M=\langle \sigma \rangle\times \langle \tau\rangle\in \mathcal{M}(S_n)$.
If, for some $c\in\mathbb{N}$ with $c<n$, there exists a $c$-cut of  $T_{\sigma}$ isolating $[1^2]$, then there exists $g\in S_n$ such that $M\leq P_x^g$, where $x=\min\{c,n-c\}.$
\end{lem}

\begin{proof}  Let $i,j\in \Omega$ such that $\tau=(i\ j)$, $\sigma(i)=i$ and $\sigma(j)=j.$  Let $[T_1\mid T_2]$ be a $c$-cut of  $T_{\sigma}$ isolating  $[1^2]$, with $c<n$.
Since passing from $c$ to $n-c$ does not change $x=\min\{c,n-c\},$ we can assume that $[1^2]$ is a subpartition of $T_1$. Let $\mathcal{O}_1(\sigma)$ be a subset of $\mathcal{O}(\sigma)$ containing the two orbits $\{i \},\ \{j \}$ and such that $T_1=T_{\mathcal{O}_1(\sigma)}$.  Then $O_1(\sigma)=\underset { X\in \mathcal{O}_1(\sigma) }{\cup}X$  has size $c$.  We show that $\sigma,\tau\in G=
\{\psi\in S_n : \psi(O_1(\sigma))=O_1(\sigma)\}$. For $\sigma$ this is a trivial consequence of the fact that $O_1(\sigma)$, by definition, is union of orbits of $\sigma$. On the other hand, since every orbit of $\tau$, up to $\{i ,j \}$, is a singleton and $O_1(\sigma)\supseteq \{i ,j \},$ we have that $O_1(\sigma)$ is also a union of orbits of $\tau.$
Thus,  $G=P_x^g,$ for a suitable $g\in S_n$.
\end{proof}

\begin{prop}\label{gamma'-g}  Let $n\in A.$ Then $\gamma'(S_n)\leq g(n).$
\end{prop}
\begin{proof}  If $n$ is odd, by Propositions \ref{odd} and \ref{gamma-g}, we immediately have $\gamma'(S_n)=\gamma(S_n)\leq g(n)$.  Let $n$ be even so that $p_1=2$ in \eqref{factors}. Let first $\nu(n)=1$, that is, $n=2^{\alpha}$, for some $\alpha\geq 2.$ By \cite[Proposition 7.5]{bp},
\[\delta=\{S_2\wr S_{2^{\,\alpha -1}},\ P_x \,:\,1\leq x<2^{\alpha-1},\   x\  \hbox{odd}\}\]
is a minimal basic set for $S_n$ of size $g(n)=2^{\alpha-2}+1$. We show that $\delta$ is special checking that every $M=\langle \sigma \rangle\times \langle \tau\rangle \in\mathcal{M}( S_n)$ is contained, up to conjugacy, in a subgroup belonging to $\delta.$  Let $T_{\sigma}=[1^2, x_1,\dots, x_k]$ be the canonical form of $T_{\sigma}$ as described in \ref{sigma}, so that for a suitable $s\in \mathbb{N}$, with $s\leq k$, $ x_i$ is even for $1\leq i\leq s$ and $  x_i$  is odd for $ s<i\leq k$.  Let $m=k-s\geq 0$.
If $m=0$, then $x_i$ is even for all $1\leq i\leq k$ and so, by Corollary \ref{wreath}, $M\leq (S_2\wr S_{2^{\,\alpha -1}})^g$, for some $g\in S_n$ and $S_2\wr S_{2^{\,\alpha -1}}\in \delta.$
 If $m> 0$, then  $ x_k$ is odd and  $n/2=2^{\,\alpha -1}$ being even, we have $x_k\neq n/2$.
Thus, if $x=\min\{x_k, n-x_k\}$, we have that $P_x\in \delta$. Moreover $T_{\sigma}=[x_k\mid 1^2, x_1,\dots,x_{k-1}]$ is a $x_k$-cut  for $T_{\sigma}$ isolating $[1^2]$. Then, by Lemma \ref{intrans},  we  get $M\leq P_x^g$, for some $g\in S_n.$

Next let $\nu(n)=2, (\alpha_1,\alpha_2)=(1,1)$, so that $n=2p_2$, with $p_2$ an odd prime. By \cite[Proposition 7.6]{bp},
$$
\delta=\{S_2\wr S_{p_2},\   P_x\
 :\ 1\leq x<n/2,\  x\  \hbox{odd}  \}
$$
is a minimal basic set for $S_n$ of size $g(n)=\frac{p_2+1}{2}$.  We need only to show that any $M=\langle \sigma \rangle\times \langle \tau\rangle\in \mathcal{M}(S_n)
$ is contained, up to conjugacy, in a subgroup belonging to $\delta.$ Let, as before,  $T_{\sigma}=[1^2, x_1,\dots, x_k]$ be the canonical form of $T_{\sigma}$. 
 If $m=0$, then $x_i$ is even for all $1\leq i\leq k$ and so, by Corollary \ref{wreath}, $M\leq (S_2\wr S_{p_2})^g$, for some $g\in S_n$. If $m> 0$, $n-2-\sum_{i=1}^sx_i$ being even, we have that $m\geq 2$ is also even. Thus $T_{\sigma}$ admits at least two odd terms among the $x_i$ for $1\leq i\leq k$. Let them be $x_u,\,x_v$. Since $x_u+x_v\leq n-2$, one of them is less than $n/2$. Calling that term $x$, we have $P_x\in \delta$. Moreover we have a cut for $T_{\sigma}=[T_1\mid T_2]$, isolating $[1^2],$ in which $T_1=[x]$ and $T_2$ is the complementary partition.
Thus, by Lemma \ref{intrans}, we deduce that $M\leq P_x^g$, for some $g\in S_n.$

Finally let $\nu(n)=2$ and  $(\alpha_1,\alpha_2)\neq (1,1)$ or $\nu(n)\geq 3$. By Proposition \ref{gamma-g}, we have that
$$
\delta=\{P_x\ :\ 1\leq x<n/2,\ \gcd(x,2p_2)=1\}\cup\{S_{2}\wr S_{n/2},\ S_{p_2}\wr S_{n/p_2}\}
$$
 is a basic set for $S_n$ of size $g(n)=\frac{n}{4}(1-\frac{1}{p_2})+2$. We show that $\delta$ is special. Let $M=\langle \sigma \rangle\times \langle \tau\rangle$ and $T_{\sigma}=[1^2, x_1,\dots, x_k]$, as above, be the canonical form of $T_{\sigma}$. 
 If $m=0$, then as in the previous cases, we get $M\leq (S_2\wr S_{n/2})^g$, for some $g\in S_n$. Next assume $m>0.$ Then for every $i\in \{s+1,\dots, k\}$, we have $x_i$ odd.
If there exists one of those $x_i$ with  $p_2\nmid x_i$, then we necessarily have $x_i\neq n/2$, because $p_2$ divides $n/2.$ It follows that
 $P_x$, with $x=\min\{x_i, n-x_i\},$ belongs to $\delta.$ By Lemma \ref{intrans}, applied to the cut of $T_{\sigma}$ with $T_1=[x]$, we then get $M\leq P_x^g$, for some $g\in S_n.$ It remains to consider the case in which $p_2\mid x_i$, for all $i\in \{s+1,\dots, k\}$. Consider the $x_i$ with $1\leq i \leq s$. We cannot have all of them divisible by $p_2$ because this would imply $p_2\mid n-(n-2)=2$, contradicting $p_2$ odd. Thus
 there exists $u\in \{1,\dots,s\}$ such that $p_2\nmid x_u.$ Pick a term $x_v$ with $v\in \{s+1,\dots, k\}$ and consider the $c$-cut $T_{\sigma}=[x_u, x_v\mid T_2]$, where $c=x_u+x_v<n$ and $T_2$ is the complementary partition of $T_1=[x_u, x_v].$ Note that the cut isolates $[1^2]$ and that $2, p_2\nmid c$. In particular $c\neq n/2$ and thus, by Lemma \ref{intrans},  defining $x=\min\{c, n-c\},$ we have $P_x\in \delta$ and  $M\leq P_x^g$, for some $g\in S_n.$

\end{proof}
\begin{cor}\label{equality}  Let $n\in A$.
If $\gamma(S_n)=g(n)$, then $r(S_n)=\gamma'(S_n)=\gamma(S_n).$
\end{cor}
\begin{proof} It is an immediate consequence of Propositions \ref{prop:prop} and \ref{equality}.
\end{proof}

\subsection{The degrees $10$ and $14$}\label{10}

We present now two interesting examples. To treat them, we need  to recall some classic  results.
\begin{lem} \label{extract} \emph{ (\cite[Theorem 13.8]{WI},  \cite[Theorem 4.11]{CA})}
A primitive group of degree $n$, which contains a permutation of type $[1^{n-m},m]$, where
$2\leq m\leq n-5,$ contains $A_n$.
\end{lem}
\begin{lem} \label{Feit} \emph{ (\cite[Lemma 3.5]{bp})}
Let $H\leq S_n$ be a primitive group containing an $n$-cycle. Then the following holds:
\begin{itemize}\item[a)] If $H$ is simply transitive or solvable, then $H\leq \mathrm{AGL}_1(p)$ with $n=p$ a prime, or $H=S_4$ with $n=4.$
\item[b)] If $H$ is nonsolvable
and doubly transitive, then one of the following holds:
\begin{itemize}
\item[i)]
 $H=S_n$ for some $n\geq5$, or $H=A_n$ for some odd $n\geq5$;
\item[ii)] $\mathrm{PGL}_d (q)\leq H\leq \mathrm{P\Gamma L}_d(q)$ and $H$ acts on $n=(q^d -1)/(q -1)$ points or hyperplanes, where $d\geq 2$ and $q$ is a prime power;
\item[iii)] $H= \mathrm{PSL}_2(11),\  M_{11}$ or $M_{23}$ with $n= 11, 11$ or $23$ respectively.
\end{itemize}
\end{itemize}

\end{lem}
\begin{lem}\label{nminusone} \emph{ (\cite[Lemma 3.8]{bp})}
Let $H$ be a group such that $\mathrm{PSL}_d(q)\leq H\leq \mathrm{P\Gamma L}_d(q),$ where $d\geq 2$ and $q$ is a prime power. Then $H,$ in its action on the  $n=(q^d-1)/(q-1)$ points or hyperplanes, contains an $(n-1)$-cycle if and only if $d=2$ and either $q$ is a prime or $(q,H)=(4,\mathrm{P\Gamma L}_2(4))$.
\end{lem}

\begin{prop}\label{10,14} $\gamma(S_{10})=\gamma'(S_{10})=3$ and $\gamma(S_{14})=\gamma'(S_{14})=4.$ Moreover, for $n\in\{10,14\}$,  no minimal basic set of $S_n$ contains $P_2$.
\end{prop}
\begin{proof} Let $n=10$. By \cite[Table 1]{bps}, we have that $\gamma(S_{10})=3=g(10)$ and thus, by Corollary \ref{equality}, we also have $\gamma'(S_{10})=3.$
Assume that there exists a minimal basic set of $S_{10}$ of type $\delta=\{P_2,H,K\}$ for some $H,K$ maximal subgroups of $S_{10}$ and $K$ transitive with $[10]\in K.$
If $A_{10}\in \delta$, then $H=A_{10}$. Consider the types $T_1=[1,4,5] $ and $T_2=[1,3,6]$. They do not belong to $P_2$ and to $A_{10}$, so they must belong to $K.$ Then $K\neq A_{10}$ contains also the type $[1^6,4]$ and so, by Lemma \ref{extract}, $K$ is imprimitive. Thus $K\in\{S_2\wr S_5, S_5\wr S_2\}$ and so $T_2$ does not belong to $K,$ a contradiction.
Assume next that $A_{10}\notin \delta$ and consider the types $T_1=[1,9], T_2=[3,7].$ Since $\gcd(3,7)=1$, by Lemma 5.2 in \cite{bp}, the only maximal subgroup of $S_{10}$,  containing a permutation of  type $T_2$ and different from $A_{10}$  is $P_3$. Thus $H=P_3.$ Since $10$ is not a prime,
 by Lemma \ref{Feit}, we deduce that $K$ is $2$-transitive. Moreover, since the only way to write $10=(q^d -1)/(q -1)$ is $10=9+1$,
 we have that $K=\mathrm{P\Gamma L}_2(9).$ On the other hand, by Lemma \ref{nminusone}, since $9$ is not a prime, such $K$ cannot contain the type $T_1$.
Since  $T_1$ does not belong also to $P_2, P_3$, we have a contradiction.

Let now $n=14$. We know from  \cite[Proposition 7.6]{bp} and from \cite[Proposition 3.4]{BBH} that $\gamma(A_{14})=4$  and that $3\leq\gamma(S_{14})\leq 4=g(14).$
We first show that $\gamma(S_{14})= 4.$
Assume, by contradiction,  that $\gamma(S_{14})=3.$ Then a minimal normal covering of $S_{14}$ must involve $A_{14}$ as a component otherwise, by intersection, we would get a normal covering of size $3$ for $A_{14}$ contradiction $\gamma(A_{14})=4.$ Let then $\delta=\{A_{14},\ H,\ K\}$ be a basic set for $S_{14}$, with $H,K$ maximal subgroups of $S_{14}$  and $K$ transitive with $[14]\in K.$ By Lemma \ref{extract}, the  types $[1,2,11],\ [3,6,5]$ cannot belong  to  a primitive subgroup of $S_{14}$. But clearly they do not belong to either of the two imprimitive maximal subgroups of $S_{14}$,  which are $S_2\wr S_7$ and $S_7\wr S_2$. So the only maximal subgroups containing them are intransitive. Since there is at most one intransitive component in $\delta$, we deduce that  $P_3\in \delta$. Since no permutation of type $T_1=[1,1,12],\ T_2=[1,4,9]$ belongs to $A_{14}$ and to $ P_3$, we deduce that $T_1,T_2\in K$ and thus $K\neq A_{14}$ is $2$-transitive and contains a $4$-cycle, contracting Lemma \ref{extract}.

By Corollary \ref{equality}, we immediately have $\gamma'(S_{14})= 4.$ We are left with showing that there exists no minimal basic set of $S_{14}$ containing $P_2$. By contradiction, assume that there exists a minimal basic set of $S_{14}$ of type $\delta=\{P_2,H,K,L\}$ for some $H,K,L$ maximal subgroups of $S_{14}$ and $L$ transitive with $[14]\in L.$ If $A_{14}\in \delta$, then let $H=A_{14}$ and note that  in $\delta$, other than $P_2$, there is  at most one intransitive component.
Consider the types $T_1=[1,3,10] $ , $T_2=[3,5,6]$ and  $T_3=[1,5,8]$. They do not belong to $P_2$ and to $A_{14}$ and, by Lemma \ref{extract}, they do not belong to primitive subgroups. On the other hand they also do not belong to $S_2\wr S_7$ and $S_7\wr S_2$. Thus they require an intransitive component. But there exists no intransitive maximal subgroup of $S_{14}$  containing all the $T_i$, for $i\in \{1,2,3\}.$
Assume next that $A_{14}\notin \delta$ and consider the types $T_1=[1,13], T_2=[3,11], T_3=[5,9], T_4=[4,10]$ Since $\gcd(3,11)=\gcd(5,9)=1$, by Lemma 5.2 in \cite{bp}, the only maximal subgroup of $S_{14}$,  containing a permutation of  type $T_2$ and different from $A_{14}$  is $P_3$. So $P_3\in \delta.$ Similarly, dealing with $T_3,$ we get that $P_5\in \delta$. Thus, we have $\delta=\{P_2,P_3,P_5, L\}$ and necessarily $T_1\in L.$ Thus, by Lemma \ref{Feit}, we get $L=\mathrm{PGL_2(13)}$. But then $5\nmid |L|$ and so no component of $\delta$ contains $T_4$ because a permutation of type $T_4$ has order $20.$
\end{proof}
The above result and Proposition \ref{odd} inspire the following
\begin{quesG2} {\em Do there exist infinitely many even $n\in\mathbb{N}$, such that no minimal basic set of $S_n$ admits $P_2$ as a component?}
\end{quesG2}

\vskip 1em

\subsection{The $h$ bound}\label{Maroti}

In this section we produce a second collection of upper bounds for $\gamma'(S_n)$, using the special basic set $\delta_1$ below, communicated to the first author by Attila Mar\'oti \cite {ma}\footnote{Mar\'oti's motivation was his proof that there are infinitely many $n$ for which $\gamma(S_n)< g(n)$, refuting \cite[Conjecture 1]{bps} which states that, for $\nu(n)\geq 2$ and $n\neq p_1p_2$, $\gamma(S_n)=g(n)$.}, and a further  basic set $\delta_2$, which turns out to be particularly useful when $n$ is odd.
These new bounds improve the
$g(n)$ bound only when $\nu(n)$ is large and $n$ is odd.
To appreciate the order of magnitude of the new bounds, we recall a number theoretical  result from \cite{bp}. For $f,\ g$ real functions defined over an upper unbounded domain, write
 $f\sim g$ if $\displaystyle{\lim_{x\rightarrow +\infty}\frac{f(x)}{g(x)}} = 1.$
Let $n\in \mathbb{N}$ and let $0\leq x<y\leq n$, with $x,y\in \mathbb{R}$. For any interval $I$ with extremes $x$ and $y$, define
$$
\phi(I;n)=|\{ i\in \mathbb{N}\ :\ i\in I,\ (i,n)=1  \}|.
$$

By \cite[Lemma 2.3 ]{bp}, we then have that  if $y-x\sim c\,n$ for some  $c>0$, then
$
\phi(I;n)\sim c\phi(n).
$

\begin{prop} \label{newbound} Let $n\in\mathbb{N}, n\geq 6$, not a prime. Then:
\begin{itemize}
\item[i)]{\rm (Mar\'oti)}  $$\delta_1=\{P_x :1\leq x\leq n/3\}\cup \{P_x:n/3< x<n/2,\,\gcd(x,n)=1\} \cup$$ $$\{S_{p_i}\wr S_{n/p_i}: i\in \{1,\dots,\nu(n)\}\}$$ is a special basic set of size $$\lfloor n/3\rfloor +\nu(n)+\phi((n/3,n/2);n)\sim  n/3 +\nu(n)+\phi(n)/6;$$
\item[ii)] if $n$ is odd, then $$\delta_2=\{P_x :1\leq x \leq n/4\}\cup \{P_x: n/4< x<n/2,\,\gcd(x,n)=1\} \cup$$$$\,\{S_{p_i}\wr S_{n/p_i}: i\in \{1,\dots,\nu(n)\}\}\cup \{A_n\}$$ is a special basic set of size $$\lfloor n/4\rfloor +\nu(n)+\phi((n/4,n/2);n)+1\sim  n/4+\nu(n)+\phi(n)/4.$$
\end{itemize}
\end{prop}
\begin{proof}  First of all, note that the condition $n$ not a prime guarantees that, for every $i\in \{1,\dots,\nu(n)\},$  $S_{p_i}\wr S_{n/p_i}\in \mathcal{W}$.

i) Since $n\geq 6$, we have that $P_2\in \delta_1$ and thus, by Corollary B, it is enough to show that $\delta_1$ is a basic set. By Corollary \ref{wreathbis}, the $n$-cycles belong, up to conjugacy, to $S_{p_1}\wr S_{n/p_1}\in \delta_1$. Consider the types $T_x=[x,n-x]$, with $1\leq x\leq n/2$. If $\gcd(x,n)=1$, then $x\neq n/2$ and $T_x$
belong to $P_x\in \delta$ ; if $\gcd(x,n)\neq1,$ then there exists $p_i$ such that $p_i\mid x$ and so, by Corollary \ref{wreathbis}, $T_x$ belongs to $S_{p_i}\wr S_{n/p_i}.$ Finally let $T=[x_1,\dots,x_k]$, with $k\geq 3$. Then at least one term is less or equal to $n/3$ and thus $T$ belongs to $P_x$ for some $1\leq x\leq n/3.$

 ii) Since $n\geq 9$, we have that $P_2\in \delta_2$ and thus, by Corollary B,  it is enough to show that $\delta_2$ is a basic set. For the types $[n]$ and $T_x=[x,n-x]$, with $1\leq x\leq n/2,$ we argue as in i).  If $T=[x_1,x_2,x_3]$, then, being  $n$ odd, $T$ belongs to $A_n$. Finally if $T=[x_1,\dots,x_k]$, with $k\geq 4$, then at least one term is less than or equal to $n/4$, and thus $T$ belongs to $P_x$ for some $1\leq x\leq n/4.$

 \end{proof}
 \begin{defn}\label{h} Define the function $h:A\rightarrow \mathbb{N}$ by
\[h(n)=\begin{cases}
\ \lfloor n/3\rfloor +\nu(n)+\phi((n/3,n/2);n) &\quad \mbox{if $n$ is even}\\

\ \lfloor n/4\rfloor +\nu(n)+\phi((n/4,n/2);n)+1    &\quad \mbox{if  $n$ is odd}\\
\end{cases}
\]
\end{defn}
\begin{cor}\label{h-bound} Let $n\in A$. Then $\gamma'(S_n)\leq h(n).$
\end{cor}
\begin{proof} If $n\geq 6$ and $n$ is not a prime, then the result follows from Proposition \ref{newbound}. If $n=4$, note that $h(4)=2=\gamma'(S_4).$ If $n=p$ is a prime we have that $h(p)=2+\frac{p-1}{2}>\gamma'(S_p)=\frac{p-1}{2}.$
\end{proof}

Depending on $n,$ we may have, in principle, $h(n)>g(n)$ or $g(n)>h(n)$ and thus, correspondingly, either the bound expressed by  Proposition \ref{gamma'-g} or that expressed by Corollary \ref{h-bound} is more strict. The next proposition shows that the bound given by Proposition \ref{gamma'-g}  is always preferable when $n$ is even. Note also that $g(4)=h(4).$
\begin{prop}\label{comparison}  Let $n\in\mathbb{N}, n\geq 5.$ If $n$ is even, then $g(n)<h(n)$. If $n$ is odd, then there exists infinitely many $n$ such that $h(n)<g(n).$
\end{prop}

\begin{proof} Let $n$ be even. If $n\leq 22$, direct computation  shows that $g(n)<h(n)$. If $n\geq 24$, we have  $g(n)\leq \frac{n}{4}(1-\frac{1}{p_2})+2<n/4+2\leq n/3\leq h(n).$

To construct an infinite family of natural numbers $n$ such that $h(n)<g(n),$ we proceed as follows.
Define $n_k=p_1p_2\cdots p_{k} $ to be the product of $k\geq 3$ consecutive  odd primes. We have $\lim_{k\rightarrow +\infty}\frac{\phi(n_k)}{n_k}=0$ as well as
$\lim_{k\rightarrow +\infty}\frac{\nu(n_k)}{n_k}=0$ and hence $\lim_{k\rightarrow +\infty}\frac{h(n_k)}{n_k}=\frac{1}{4}.$ On the other hand, we have $\lim_{k\rightarrow +\infty}\frac{g(n_k)}{n_k}= \frac{1}{2}(1-\frac{1}{p_1}) (1-\frac{1}{p_2})$. Choose now the  primes $p_1, p_2$ such that $(1-\frac{1}{p_1}) (1-\frac{1}{p_2})>\frac{1}{2}$. With this choice of $p_1,p_2$, we get  $\lim_{k\rightarrow +\infty}\frac{g(n_k)-h(n_k)}{n_k}>0$ and thus $h(n_k)<g(n_k)$  for an infinite number of $n_k$.

 \end{proof}

\vskip 1em
\subsection{The even degree case}\label{even-case} We focus now on the case $n$ even, trying to shed light on the Group-theoretic Question 2.
To appreciate what we are going to show, note that the canonical basic set $\delta_C$ in \eqref{basic-set} does not contain $P_2.$

\begin{prop}\label{even-new}  Let $n\in A$  be even. Then the set
\[\delta_E=\{P_x :1\leq x < n/2,\  2\mid x\}\cup \,\{S_{n/2}\wr S_{2}\}\cup \{A_n\}\]
is a special basic set of size $|\delta_E|=\left\lceil  \frac{n+4}{4}\right\rceil.$ Moreover, $g(n)\leq |\delta_E|$ with equality if and only if $n=2^{\alpha}$, for some $\alpha\in \mathbb{N}$ with $\alpha\geq 2,$ or $n=4q$, for some prime $q$.

\end{prop}

\begin{proof}   Let $n$ be even. We show that $\delta_E$ is a special basic set for $S_n.$ Since $P_2\in \delta_E,$  by Proposition \ref{prop:prop}, it is enough to show that $\delta_E$ is a basic set, that is, all the types of $S_n$ appear in some component of $\delta_E$.
By Corollary \ref{wreathbis}, the type $[n]$ belongs to $\in S_{n/2}\wr S_{2}$.  The types with an even number of parts belong instead to $A_n$. Let $T=[x_1,\dots,x_k]$ be a type with an odd number $k\in\mathbb{N}$ of parts. Then, there exists an even term $x_i$. If  $x_i\neq n/2,$ then consider $x=\min\{x_i,n-x_i\}<n/2$ and note that $T\in P_x.$
If instead the only even term is equal to $n/2$, then $T\in S_{n/2}\wr S_{2}$.
The size of $\delta_E$ is clear counting the even numbers in $[1,n/2)\cap \mathbb{N}.$

We show that $|\delta_E|\leq g(n)$, taking into account the cases of the definition  \eqref{g-definition}  of $g(n).$ Let first $4\mid n$, so that $|\delta_E|=\frac{n}{4}+1$. If $n=2^{\alpha}$ for some $\alpha\in \mathbb{N}$ with $\alpha\geq 2$ we have $g(n)=\frac{n}{4}+1$. If  $\nu(n)\geq2$ and $n=2^{\alpha_1}\cdots p_{\nu(n)}^{\alpha_{\nu(n)}}$, for some $\alpha_1\geq 2$ and $ \alpha_i\geq 1$ for $i\in\{1,\dots,\nu(n)\},$ then $g(n)=\frac{n}{4}(1-\frac{1}{p_2})+2\leq \frac{n}{4}+1.$ Moreover, it is immediately checked that $g(n)=\frac{n}{4}+1$ holds only if $n=4p_2.$ Consider next the case $4\nmid n$ so that $|\delta_E|=\frac{n+6}{4}$. Since $n\geq 4$, we have $\nu(n)\geq2$ and $n=2p_2^{\alpha_2}\cdots p_{\nu(n)}^{\alpha_{\nu(n)}}$, with $ \alpha_i\geq 1$ for $i\in\{1,\dots,\nu(n)\}.$
If $\nu(n)=2$ and $\alpha_2=1$, we have $n=2p_2$ and $g(n)=\frac{n+2}{4}=|\delta_E|-1<|\delta_E|.$ If instead $\alpha_2\geq 2$ or $\nu(n)\geq 3,$ we have that $\frac{n}{p_2}>2$ and thus $g(n)=\frac{n}{4}(1-\frac{1}{p_2})+2<\frac{n+6}{4}.$
 \end{proof}

\begin{cor}\label{application} Let $n\in A$  be even of the type $n=2^{\alpha}$, for some $\alpha\in \mathbb{N}$ with $\alpha\geq 2,$ or $n=4q$ for some prime  $q$. If $\gamma(S_n)=g(n),$ then there exists a minimal basic set of $S_n$ containing $P_2$ as a component.
\end{cor}
\begin{proof}  By $\gamma(S_n)=g(n),$ we get that the basic set $\delta_E$ in Proposition \ref{even-new} is minimal. Moreover, by definition, $P_2\in\delta_E.$
\end{proof}
Recall that we do not know any infinite family of even  $n$ such that $\gamma(S_n)=g(n).$ Thus the above result does not imply the existence of infinitely many even $n$ such that $P_2$ appears as a component in some minimal basic set.
Note that Proposition \ref{even-new} implies that Conjecture 3 in \cite{bps}, stating that  for each  minimal
basic set of $S_n$, with $\nu(n)\geq 2$, consisting of maximal components, the subset of intransitive
components is given by $\{P_x\in \mathcal{P}\ :  \  \gcd(x,p_1p_2)=1\}$, is generally false for $n$ even. Indeed we have $\gamma(S_{12})=4=g(12)$ and in the intransitive components $P_x\in\delta_E$, $x$ is even.

We close this section observing, as a further consequence of Proposition \ref{even-new}, that no kind of uniqueness seems possible for the minimal special basic sets in the even case. For instance the three sets of subgroups of $S_8$ given by
\[\delta_E=\{P_2, A_8, S_{4}\wr S_{2} \},\quad \delta_C=\{P_1, P_3, S_{4}\wr S_{2}\}, \quad \delta=\{P_1, A_8, S_{2}\wr S_{4}\}\]
are all minimal special basic sets for $S_8.$  This follows from $\gamma(S_8)=3=g(8)$, using Corollary \ref{equality}  for  $\delta_C$ and using  Proposition \ref{even-new} for  $\delta_E.$ For the set $\delta$ the check is easily carried on by the usual arguments.

\vskip 1cm

\section{Linear bounds for $r(S_n)$ and $\gamma'(S_n)$ }\label{6}
In this last section we derive for the parameter $r(S_n)$ and $\gamma'(S_n)$, the same linear bounds \eqref{linear} known for $\gamma(S_n)$.
\begin{prop}\label{bounds}  Let $n\in\mathbb{N}, n\geq 3.$ Then there exists a positive constant $k$ such that $kn\leq r(S_n)\leq \gamma'(S_n)\leq 2n/3$.
\end{prop}
\begin{proof}  If $n=3$, we use \eqref{3} and $k\leq 2/3.$
Let  $n\geq 4.$ By Proposition \ref{prop:prop} and Proposition \ref{gamma'-g}, we know that  $$\gamma(S_n)\leq r(S_n)\leq
 \gamma'(S_n) \leq g(n).$$ Since, by \cite[Theorem 1.1] {JA}, $\gamma(S_n)\geq kn,$ we immediately derive the lower bound $kn\leq r(S_n)\leq \gamma'(S_n).$ To deal with the upper bound, consider the function $g(n)$. Since the function $f(x)=1-1/x$ is increasing for $x>0$, we have that $g(n)<n/2+2$, and  $g(n)$ being an integer says $g(n)\leq \frac{n+3}{2}.$ If $n\geq 9$,  we note that $\frac{n+3}{2}\leq 2n/3$.  The cases $4\leq n\leq 8$ are all included in \cite[Table 1]{bps}  and realise
 $\gamma (S_n)=g(n)$, so that by Corollary \ref{equality} and  \cite[Theorem 1.1] {JA}, we also have $\gamma'(S_n)=\gamma(S_n)\leq 2n/3$.
\end{proof}
In conclusion we note that the  constant $k$  relies on certain number theoretic results (see \cite{BLS})  and that the known value of $k$ is unrealistically small
because of the many approximations needed first to obtain and next to apply those results. For instance, we know that for $n\geq 792,000$ even, $k=0.025$ works (\cite[Remark 6.5.]{JA}).
We still do not know optimum values for the constant $k$.
\section{
Acknowledgments}
The authors wish to thank Attila Mar\'oti for his permission to use his personal communication, which led to Section \ref{Maroti}.
The first author is supported by GNSAGA (INdAM).


\bibliographystyle{plain}

\def\cprime{$'$}

\end{document}